\documentclass[12pt,reqno]{amsart}
\usepackage[left=1in,right=1in,top=1in,bottom=1in]{geometry}
\usepackage{amssymb}
\usepackage{amsmath}
\usepackage{mathrsfs}
\usepackage{amsfonts}
\usepackage{hyperref}
\usepackage{xcolor}
\usepackage{bbm}
\usepackage{tikz}
\usepackage{array}
\usepackage{multirow}
\usepackage{tikz}
\usepackage{ytableau}
\usepackage{thmtools}
\usepackage{thm-restate}
\usepackage{varwidth}
\usepackage{comment}
\usepackage{mathtools}
\usepackage[mathscr]{euscript}
\usepackage{bbm}
\usepackage{multicol}
\usepackage{enumerate}
\usepackage{graphicx} 
\theoremstyle{definition}
\newtheorem{thm}{Theorem}[section]
\newtheorem{lem}[thm]{Lemma}
\newtheorem*{lem*}{Lemma}
\newtheorem*{thm*}{Theorem}
\newtheorem{prop}[thm]{Proposition}

\newtheorem{cor}[thm]{Corollary}

\newtheorem{defn}[thm]{Definition}
\newtheorem*{remark*}{Remark}
\newtheorem{remark}{Remark}
\newtheorem{example}{Example}

\newtheorem{cor/defn}[thm]
{Corollary/Definition}

\DeclareMathOperator{\Par}{\mathbb{Y}}

\DeclareMathOperator{\GL}{\mathrm{GL}}

\DeclareMathOperator{\Hom}{\mathrm{Hom}}

\usepackage[
    backend=bibtex,
    maxnames=5,
    style=alphabetic
  ]{biblatex}
\title{Higher Order Bell Symmetric Functions}

\author{Milo Bechtloff Weising}
\address{Department of Mathematics (0123),
460 McBryde Hall, Virginia Tech,
225 Stanger Street,
Blacksburg, VA 24061-1026}
\email{milojbw@vt.edu}

\date{\today}

\begin{document}

\maketitle 

\begin{abstract}
    We study symmetric function analogues of the higher order Bell numbers. Their construction involves iterated plethystic exponential towers mimicking the single variable exponential generating functions for the higher order Bell numbers. We derive explicit recurrence relations for the expansion coefficients of the Bell functions into the monomial and power sum bases of the ring of symmetric functions. Using the machinery of combinatorial species, the Bell functions are proven to be the Frobenius characteristics of the permutation representations of symmetric groups on hyper-partitions of certain orders and sizes. In the order 1 case, we are able to give more details about the expansion coefficients of the Bell functions in terms of vector partitions and divisor sums as well as give a recurrence relation analogous to the well known recursion for the Bell numbers. Lastly, we use Littlewood's reciprocity theorem and the Hardy-Littlewood Tauberian theorem to prove that the Schur expansion coefficients of the order 1 Bell functions are certain asymptotic averages of restriction coefficients.
\end{abstract}
\tableofcontents

\section{Introduction}

The operation of plethysm plays an important role
in the study of symmetric functions. Often, plethystic compositions are used to algebraically encode complicated combinatorial or representation-theoretic symmetric functions in a simpler way. Despite the ubiquity of plethysms in algebraic combinatorics, we do not yet understand plethsytic substitutions combinatorially. One of the longest standing open problems in combinatorics may be phrased in this manner. The \textbf{plethysm problem}, which remains largely open, asks for a combinatorial interpretation for the Schur expansion coefficients of the Schur function plethysms $s_{\lambda}(s_{\mu}(X)).$ Other problems like the \textbf{restriction problem} have a similar interpretation. 

\textbf{Schur positive} symmetric functions are central to the study of symmetric functions. This is, in part, due to their representation theoretic importance as combinatorial analogues for characters of representations of symmetric groups. All of the main operations of symmetric functions (addition, multiplication, Kronecker product, plethysm, etc.) preserve the set of Schur positive symmetric functions. As such, it is easy to define families of symmetric functions using these operations which are manifestly Schur positive, but hard to understand combinatorially. In this paper, we study one such family of symmetric functions. In particular, we study the \textbf{plethystic exponential towers} 
$$\Omega(X),\Omega(\Omega(X)-1),\Omega(\Omega(\Omega(X)-1)-1),\ldots.$$ These are generalizations of the single-variable exponential towers $\exp(t),\exp(\exp(t)-1),\ldots .$ When interpreted as exponential generating functions, the latter power series encode the \textbf{higher order Bell numbers} $b_n^{(m)}$. As such, the homogeneous components $B_n^{(m)}(X)$ of the plethystic exponential towers are symmetric function analogues of the higher order Bell numbers. 

This paper is structured as follows. First, we review the basics of symmetric functions and some other combinatorial definitions. Next, we define the higher Bell functions $B_n^{(m)}$ and prove some immediate results relating the $B_n^{(m)}$ to the Bell numbers $b_n^{(m)}$ and showing their Schur positivity. Then we use the machinery of combinatorial species to prove that the $B_n^{(m)}$ are the Frobenius characteristics arising from the $\mathfrak{S}_n$ action on the order $m$ \textbf{hyper-partitions} of $\{1,\ldots,n\}.$ This explains both the connection with the Bell numbers $b_n^{(m)}$ and the Schur positivity. Afterwards, we algebraically determine recursive formulae for the monomial and power sum expansion coefficients of the Bell functions. We then focus entirely on the $m=1$ case of the Bell functions $B_n^{(1)}$, giving more detail on their structure, and proving a recursive formula relating $B^{(1)}_{n+1}$ to $B^{(1)}_{i}$ for $0\leq i \leq n$. In the last section, we show that the Schur expansion coefficients of the functions $B_n^{(1)}$ are asymptotic averages of certain \textbf{ restriction coefficients}. 

\subsection{Symmetric functions}

Here we briefly introduce some notation and results related to symmetric functions which we will need later in the paper. We refer the reader to Macdonald's book \cite{Macdonald} for a more complete introduction to symmetric functions. 

\begin{defn}
    Let $\Lambda$ denote the graded completion of the \textit{\textbf{ring of symmetric functions}} over $\mathbb{Q}.$ As an abuse of terminology, we will call elements of $\Lambda$ symmetric functions. Elements of $\Lambda$ have the form $\sum_{n \geq 0} F_n(X)$ where $F_n(X)$ is a homogeneous symmetric function of degree $n$. We will write $\Lambda_0$ for all $\sum_{n \geq 0} F_n(X) \in \Lambda$ with $F_0 = 0$ and $\Lambda'$ for the subalgebra of all symmetric functions $\sum_{n \geq 0} F_n(X)$ where $F_n(X) = 0$ for all but finitely many $n.$ Define $m_{\lambda},h_{\lambda},p_{\lambda},s_{\lambda}$ respectively for the \textit{\textbf{monomial, complete homogeneous, power sum, and Schur functions}} indexed by partitions $\lambda \in \Par.$ A symmetric function $F \in \Lambda$ is \textit{\textbf{Schur positive}} if it has non-negative integer coefficients when expanded in the Schur basis $s_{\lambda}.$
\end{defn}

Symmetric functions appear naturally in the study of representations of the \textbf{symmetric groups} $\mathfrak{S}_n$ over $\mathbb{C}.$

\begin{defn}
    For $n \geq 1$ let $\mathfrak{S}_n$ denote the group of bijections of the set $\{1,\ldots,n\}.$ Given a finite dimensional complex representation $V$ of $\mathfrak{S}_n$ define the \textit{\textbf{Frobenius characteristic}} of $V$ as 
    $$\mathrm{Frob}_{\mathfrak{S}_n}(V) = \frac{1}{n!} \sum_{\sigma \in \mathfrak{S}_n} \mathrm{Tr}_{V}(\sigma) p_{\lambda(\sigma)}(X) \in \Lambda$$ where $\lambda(\sigma)$ is the cycle-type partition of $\sigma.$
\end{defn}

\begin{defn}
    The \textit{\textbf{Hall inner product}} $\langle - , - \rangle: \Lambda' \times \Lambda \rightarrow \mathbb{Q}$ is the pairing defined by 
    $$\left\langle \sum_{\lambda}a_{\lambda}p_{\lambda} , \sum_{\mu} b_{\mu} p_{\mu}\right \rangle:= \sum_{\lambda} a_{\lambda} b_{\lambda} z_{\lambda}$$ where 
    $$z_{\lambda}:= \prod_{j \geq 1} \left( m_j(\lambda)! j^{m_{j}(\lambda)} \right)$$ and $m_j(\lambda)$ denotes the number of $j$'s appearing in the partition $\lambda.$
\end{defn}

Throughout this paper, we will use \textbf{plethystic substitution} of symmetric functions. For the purposes of this paper, we will define plethystic substitution as follows:

\begin{defn}
    Let $t_1,t_2,\ldots$ be a set of commuting free variables. For a formal series $A \in \mathbb{Q}[[t_1,t_2,\ldots ]]$, say $A= \sum_{\alpha} c_{\alpha} t^{\alpha}$ and $n \geq 1$ define 
    $$p_n(A) = \sum_{\alpha} c_{\alpha} t^{n\alpha}.$$ Now we define the plethystic substitution $F \mapsto F(A)$ as the (non-unital) $\mathbb{Q}$-algebra homomorphism $\Lambda_0 \rightarrow \mathbb{Q}[[t_1,t_2,\ldots ]]$ uniquely determined by $p_n \mapsto p_n(A).$ For $A(t) \in \mathbb{Q}[[t_1,t_2,\ldots ]]$ with $A(0) = 0$ we may also define the substitution $F \mapsto F(A)$ as the unique algebra homomorphism extending the plethystic substitution $\Lambda_0 \rightarrow \mathbb{Q}[[t_1,t_2,\ldots ]]$ by mapping $1 \mapsto 1.$ In particular, this defines plethsytic substitution for symmetric functions as a map $\Lambda \times \Lambda_0 \rightarrow \Lambda.$ We may also define substitutions of the form $\Lambda' \times \Lambda \rightarrow \Lambda$ using the identity $p_n(1+A) = 1 + p_n(A)$ and extending accordingly. 
\end{defn}

\begin{remark}
    We will be careful throughout this paper to avoid any plethystic substitutions of the form $F(G)$ where $F \notin \Lambda'$ and $G \notin \Lambda_0.$ We will leave such substitutions undefined.  
\end{remark}

\begin{defn}
Define the \textit{\textbf{plethystic exponential}} as
    $$\Omega(X):= \sum_{n \geq 0} h_n(X).$$ It will be useful to define $$\Omega_0(X) := \Omega(X) -1 = \sum_{n \geq 1} h_n(X).$$
\end{defn}

We will be interested in repeated plethystic substitutions of the plethystic exponential $ \Omega_0(X)$. 

\begin{defn}
We define the symmetric functions $\Omega^{(m)}_0(X) \in \Lambda_0$ inductively by
    $$\Omega_0^{(0)}(X) := h_1(X)$$ and for $m \geq 0$
    $$\Omega_0^{(m+1)}(X) := \Omega_0\left( \Omega_0^{(m)}(X)\right).$$
\end{defn}

Importantly, note that $\Omega^{(m)}(X)$ is indeed well-defined as a symmetric function. If instead we had considered plethystic substitutions of the form $\Omega(\Omega(X))$, we would not obtain a well-defined symmetric function since $\Omega(1)$ itself is not well-defined. The following are well-known results regarding the plethystic exponential.

\begin{lem}\label{plethystic exp lem}
    
The plethystic exponential satisfies the following properties:
\begin{enumerate}
    \item $\Omega(X) = \prod_{i \geq 1} (1-x_i)^{-1} = \exp \left( \sum_{n \geq 1} \frac{p_n(X)}{n} \right)$
    \item $\Omega(X+Y) = \Omega(X)\Omega(Y)$
    \item $\Omega(XY) = \sum_{\lambda} \frac{p_{\lambda}(X)p_{\lambda}(Y)}{z_{\lambda}} = \sum_{\lambda}s_{\lambda}(X)s_{\lambda}(Y).$
\end{enumerate}
\end{lem}

We will need the following specialization map.

\begin{defn}
    Define the homomorphism $\psi: \Lambda \rightarrow \mathbb{Q}[[p_1]]$ by $\psi(p_m):= 0$ if $m > 1$ and $\psi(p_1):= p_1.$ 
\end{defn}

The specialization $\psi$ satisfies the following properties:

\begin{lem}\label{psi exp lem}
    $$\psi(\Omega(X)) = \exp(p_1)$$
\end{lem}
\begin{proof}
    This follows from the identity 
    $\Omega(X) = \exp\left( 
\sum_{n \geq 1} \frac{p_n}{n} \right).$
\end{proof}

\begin{lem}\label{psi comp lem}
For $F\in \Lambda$ and $G \in \Lambda_0$, 
    $$\psi( F(G(X)) ) =  \psi(F) ( \psi(G)(p_1)).$$
\end{lem}
\begin{proof}
    This follows from the identity $p_n(p_m(X)) = p_{nm}(X)$ for all $n,m \geq 1.$
\end{proof}

\subsection{Bell numbers}

\begin{defn}\cite{Bell}
    For $m \geq 0$ define the power series $E^{(m)}(t) \in \mathbb{Q}[[t]]$ by $E^{(0)}(t) = \exp(t) := \sum_{n \geq 0} \frac{t^n}{n!}$ and $E^{(m+1)}(t):= E^{(m)}(\exp(t)-1)$ for $m \geq 0.$ For $n,m \geq 0$ the \textit{\textbf{higher order Bell numbers}} $b_n^{(m)}$ are defined by 
    $$E^{(m)}(t) = \sum_{n \geq 0} \frac{b_n^{(m)}}{n!} t^n.$$
\end{defn}

For $m =1$ the $b_n^{(1)}$ are the usual Bell numbers counting the number of set partitions of $\{1,\ldots,n\}.$ It is known that the higher order Bell numbers $b_n^{(m)}$ count the number of order m hyper-partitions of $\{1,\ldots,n\}$ \cite{SK19}.

\section{Higher order Bell functions}
\subsection{Construction}

In this section, we will construct symmetric function analogues of the higher order Bell numbers and discuss some of their properties.

\begin{defn}
For $m \geq 0$ define the \textit{\textbf{higher order Bell symmetric functions}} $B_n^{(m)}(X) \in \Lambda$ by 
    $$\Omega\left( \Omega_0^{(m)}(tX) \right) = \sum_{n \geq 0} B_n^{(m)}(X)t^n.$$
\end{defn}

In essence, we are simply replacing each series $\exp(t)$ along with the usual power series composition with the plethystic exponential $\Omega(X)$ and plethystic substitution. The $m =1$ case of the higher order Bell symmetric functions have occurred previously in the literature, for example in Gessel's paper \cite{Gessel}, but to the author's knowledge have not been studied directly in their own right as symmetric functions.

\begin{remark}
    The $0$-th order Bell functions are the complete homogeneous symmetric functions $B^{(0)}_n = h_n$ and the corresponding Bell numbers are simply $b_n^{(0)} = 1.$ Note that for all $m \geq 0,$ $B_0^{(m)} = 1.$
\end{remark}

We may directly relate the higher order Bell symmetric functions to the higher order Bell numbers as follows: 

\begin{prop}\label{Bell numbers from sym functions}
For all $n,m \geq 0$,
    $b_n^{(m)} = \langle B_n^{(m)},p_{1^n} \rangle.$
\end{prop}
\begin{proof}
    Using Lemmas \ref{psi exp lem} and \ref{psi comp lem} inductively, we see that 
    $\Psi\left(\Omega \left( \Omega_0^{(m)}(X) \right) \right) = E^{(m)}(p_1) = \sum_{n \geq 0} \frac{b_n^{(m)}}{n!} p_1^n$ so that 
    $\Psi(B_n^{(m)}) = \frac{b_n^{(m)}}{n!} p_1^n.$ Note that for $F$ homogeneous of degree $n$, $\langle F, p_1^n \rangle = \langle \Psi(F), p_1^n \rangle.$ Therefore, 
    $\langle B_n^{(m)},p_{1^n} \rangle = \langle \Psi(B_n^{(m)}),p_{1^n} \rangle =\langle \frac{b_n^{(m)}}{n!} p_{1^n},p_{1^n} \rangle = b_n^{(m)} \frac{\langle p_{1^n},p_{1^n} \rangle }{n!} = b_n^{(m)}.$
\end{proof}

\subsection{Schur positivity}

Here we will show that the Bell functions $B_n^{(m)}$ are Schur positive. We will first do so algebraically and then later we will show that the $B_n^{(m)}$ are Frobenius characteristics of certain $\mathfrak{S}_n$ representations. We first need the following:

\begin{lem}\label{recurrence lemma}
    For all $n,m \geq 0,$ 
    $B_n^{(m+1)} = \sum_{\lambda \vdash n} \prod_{j \geq 1} h_{m_j(\lambda)}(B_{n}^{(m)}).$
\end{lem}
\begin{proof}
We proceed by direct computation:
\begin{align*}
    \Omega\left( \Omega^{(m+1)}_0(tX) \right) &= \Omega(\Omega_0(\Omega_0^{(m)}(tX)) )\\
    &= \Omega((\Omega(\Omega_0^{(m)}(tX))-1 ))\\
    &= \Omega \left(\sum_{n \geq 1} B_{n}^{(m)}(tX) \right)\\
    &= \prod_{n \geq 1} \Omega \left( B_n^{(m)}(tX) \right)\\
    &= \prod_{n \geq 1}\left( \sum_{k \geq 0}  h_k\left(B_n^{(m)}(tX) \right) \right).\\
\end{align*}
Lastly, we expand this product into a series
\begin{align*}
    &= \sum_{\substack{k_1,k_2,\ldots \\
    k_n = 0 ~a.e.}}  \prod_{n \geq 1} h_{k_n}\left( B_{n}^{(m)}(tX) \right) \\
    &= \sum_{\substack{k_1,k_2,\ldots \\
    k_n = 0 ~a.e.}} t^{\sum_{n \geq 1} nk_n}  \prod_{n \geq 1} h_{k_n}\left( B_{n}^{(m)}(X) \right) \\
    &= \sum_{\lambda} t^{|\lambda|} \prod_{n \geq 1} h_{m_n(\lambda)}\left(B_n^{(m)}(X) \right)\\
    &= \sum_{n \geq 0} t^n \sum_{\lambda \vdash n} \prod_{j \geq 1} h_{m_j(\lambda)}\left(B_{j}^{(m)}(X)\right).\\
\end{align*}

\end{proof}

The Bell functions $B_n^{(m)}$ are manifestly Schur positive.

\begin{prop}\label{plethystic recursion}
    For all $n,m \geq 0,$ the higher order Bell function $B_n^{(m)}$ is Schur positive. 
\end{prop}
\begin{proof}
    The proof follows by induction on the order $m \geq 0.$ For $m =0,$ $B_n^{(0)} = h_n=s_{(n)}$ is clearly Schur positive. Now assume for some $m \geq 0$ that every $B_n^{(m)}$ is Schur positive. Using Lemma \ref{recurrence lemma}, we see that for all $n \geq 0$
    $$B_n^{(m+1)} = \sum_{\lambda \vdash n} \prod_{j \geq 1} h_{m_j(\lambda)}(B_{j}^{(m)}).$$ By assumption, each $B_{n}^{(m)}$ is Schur positive so that the plethysms $h_{m_j(\lambda)}(B_{n}^{(m)})$ are also Schur positive and thus the products $\prod_{j \geq 1} h_{m_j(\lambda)}(B_{j}^{(m)})$ are Schur positive. Thus each $B_n^{(m+1)}$ is Schur positive as required.
\end{proof}

\begin{example}

Below is a table showing the Schur expansions of the first few Bell functions:

\begin{center}
\begin{tabular}{ | m{1.5cm} | m{1.5cm}| m{1.5cm} | m{3cm} | m{5cm} | } 
  \hline 
  $B_n^{(m)}$ & $n=1$  & $n=2$  & $n =3$  & $n =4$ \\ 
  \hline
  $m = 0$ & $s_1$ & $s_2$ & $s_3$ & $s_4$  \\
  \hline
  $m=1$ & $s_1$ & $2s_2$ & $ s_{2,1} + 3s_3$ & $2s_{2,2} + 2s_{3,1} + 5s_4$  \\ 
  \hline
  $m=2$ & $s_1$ & $ 3s_2 $ & $3s_{2,1} + 6s_3$ & $s_{2,1,1} + 8s_{2,2} + 9s_{3,1} + 14s_{4}$ \\ 
  \hline
  $m=3$ & $s_1$ & $4s_2$ & $6s_{2,1} + 10s_3$ & $4s_{2,1,1} + 20s_{2,2} + 24s_{3,1} + 30s_{4}$ \\
  \hline
  $m=4$ & $s_1$ & $5s_2$ & $10s_{2,1} + 15 s_3$ & $10s_{2,1,1}+40s_{2,2}+50s_{3,1} + 55s_4$ \\
  \hline
\end{tabular}
\end{center}

\end{example}

\begin{remark}
    The Schur expansion coefficients of the Bell functions appear to satisfy monotonicity with respect to the \textbf{dominance ordering}. It would be interesting to see if this in fact holds in general. Furthermore, it is an interesting problem to find combinatorial formulas for these coefficients. From Proposition \ref{plethystic recursion} we see that if we knew combinatorial formulas for plethysm coefficients $s_{\lambda}(s_{\mu})$ (which we do not), then we could recursively determine the Schur expansions of the $B_n^{(m)}$ albeit in a roundabout manner. 
\end{remark}

\subsection{Combinatorial species and hyper-partitions}

Using the machinery of \textbf{combinatorial species}, we may identify specific representations whose Frobenius characteristics are the higher order Bell functions. Here we give an abridged introduction to combinatorial species. We refer the reader to Bergeron--Labelle--Leroux's book \cite{BLL98}.

\begin{defn}
    Define $\mathbb{B}$ to be the category whose objects are finite sets and whose morphisms are bijective functions. A  \textit{\textbf{species}} is a functor $\mathcal{F}: \mathbb{B} \rightarrow \mathbb{B}.$ For all $n \geq 0$ we use the shorthand $\mathcal{F}[n]:= \mathcal{F}(\{1,\ldots, n\}).$ Define the species of sets $\mathcal{S}$ as $\mathcal{S}(U) := \{U\}$ and for a bijection $f:U \rightarrow V,$ $\mathcal{S}(f)( \{U\}):= \{V\}.$ Define the species $\mathcal{S}_{+}$ of non-empty sets to agree with $\mathcal{S}$ for all finite non-empty sets and $\mathcal{S}_{+}(\emptyset) := \emptyset.$ Define the species of set partitions $\mathcal{P}$ by 
    $\mathcal{P}(U):= \{ \pi = \{W_1,\ldots,W_k\} ~|~ W_{i} \neq \emptyset,~ \sqcup_{W \in \pi} W = U\}$ and for a bijection $f:U \rightarrow V$ define $\mathcal{P}(f)( \{W_1,\ldots,W_k\}):= \{f(W_1),\ldots,f(W_k)\}.$
\end{defn}

 The main operation on species which we will be interested in is composition. 

\begin{defn}
    If $\mathcal{F},\mathcal{G}$ are species, then we define the composition species $\mathcal{F}\circ \mathcal{G}$ by 
    $$(\mathcal{F}\circ\mathcal{G})(U):= \bigcup_{\pi \in \mathcal{P}(U)} \left( \mathcal{F}(\pi) \times \prod_{W \in \pi} \mathcal{G}(W) \right)$$ and for a bijection $f:U \rightarrow V$ define $(\mathcal{F}\circ\mathcal{G})(f):(\mathcal{F}\circ\mathcal{G})(U) \rightarrow (\mathcal{F}\circ\mathcal{G})(V)$ using $\mathcal{F}(\mathcal{P}(f))$ and $\mathcal{G}(f|_{W})$ accordingly. 
\end{defn}

As an example, it is instructive to check that $\mathcal{P} = \mathcal{S}\circ \mathcal{S}_{+}.$ Every species defines a symmetric function in the following way:

\begin{defn}
    For a species $\mathcal{F}$ define the \textit{\textbf{cycle index series}} $Z_{\mathcal{F}} \in \Lambda$ as 
    $$Z_{\mathcal{F}}:= \sum_{\lambda \in \Par} \frac{\mathrm{fix}(\mathcal{F}[n],\lambda)}{z_{\lambda}} p_{\lambda}(X)$$ where $\mathrm{fix}(\mathcal{F}[n],\lambda)$ is the number of fixed points in $\mathcal{F}[n]$ of any permutation $\sigma \in \mathfrak{S}_n$ with cycle type $\lambda.$
\end{defn}

For all $n \geq 1$ the homogeneous degree $n$ component of $Z_{\mathcal{F}}$ is the Frobenius characteristic of the $\mathfrak{S}_n$ representation determined by the group action of $\mathfrak{S}_n$ on $\mathcal{F}[n].$ As an example, we have that $Z_{\mathcal{S}} = \Omega(X)$ and $Z_{\mathcal{S}_+} = \Omega_0(X).$

\begin{defn}
    For $m \geq 0$ define the species of order $m$ hyper-partitions $\mathcal{P}^{(m)}$ by $\mathcal{P}^{(0)}:= \mathcal{S}$ and $\mathcal{P}^{(m+1)} = \mathcal{P}^{(m)}\circ \mathcal{S}_{+}$ for all $m \geq 0.$ We will write $\pi^{(m)}_n:= \mathcal{P}^{(m)}[n]$ for the set of order $m$ \textit{\textbf{hyper-partitions}} of $\{1,\ldots,n\}.$
\end{defn}

\begin{lem}\cite{BLL98}\label{species comp lem}
    If $\mathcal{F}$ and $\mathcal{G}$ are combinatorial species with $\mathcal{G}(\emptyset) = \emptyset,$ then 
    $$Z_{\mathcal{F} \circ \mathcal{G}} = Z_\mathcal{F}(Z_\mathcal{G}).$$
\end{lem}

The main result in this section is the following:

\begin{thm}\label{Frob thm}
    For all $n,m \geq 0$, $B_n^{(m)}(X)$ is the Frobenius characteristic of the $\mathfrak{S}_n$ action on the set $\pi_n^{(m)}$ of order $m$ hyper-partitions of $\{1,\ldots, n\}.$
\end{thm}
\begin{proof}
    By induction using Lemma \ref{species comp lem}, we see that $\Omega( \Omega^{(m)}_0(X))$ is the cycle index series for the species $\mathcal{P}^{(m)}$. The result follows.
\end{proof}

\subsection{Monomial expansion}

We will provide a recursive formula for the coefficients in the monomial expansions of the Bell functions. 

\begin{defn}
    Define $Q$ as the subset of $\mathbb{Z}_{\geq 0}^{\infty}$ consisting of finite support sequences. Let $Q' := Q \setminus \{0\}.$ We consider $\mathbb{Y} \subset Q$ as the set of weakly decreasing vectors. Define $\rho^{(m)}(\lambda) \in \mathbb{Q}$ by
    $B^{(m)}_n = \sum_{\lambda \vdash n} \rho^{(m)}(\lambda) m_{\lambda}.$ Extend $\rho^{(m)}$ to $\alpha \in Q$ by setting $\rho^{(m)}(\alpha):= \rho^{(m)}(\lambda)$ where $\lambda$ is the weakly decreasing rearrangement of $\alpha.$
\end{defn}

Note that since $B^{(0)}_n = h_n$ for all $n \geq 0,$ we know that $\rho^{(0)}(\lambda) = 1$ for all $\lambda \in \Par.$ The numbers $\{ \rho^{(m)}(\lambda) \}_{m\geq 0, \lambda \in \Par}$ satisfy the following recursion:

\begin{prop}\label{monomial recursion prop} For all $m \geq 0$ and $\lambda \in \Par \setminus \{ \emptyset \},$
    $$\rho^{(m+1)}(\lambda) = \sum_{\substack{\phi:Q'\rightarrow \mathbb{Z}_{\geq 0} \\ \sum_{\alpha \in Q'} \phi(\alpha)\alpha = \lambda}} \prod_{\alpha \in Q'} \binom{\phi(\alpha) + \rho^{(m)}(\alpha)-1}{\rho^{(m)}(\alpha)-1}.$$
\end{prop}
\begin{proof}
By direct computation,
$$\Omega\left( \Omega^{(m+1)}_0(X) \right)=\Omega\left( \sum_{n \geq 1} B_n^{(m)}(X) \right)= \Omega\left( \sum_{\alpha \in Q'} \rho^{(m)}(\alpha) x^{\alpha} \right)= \prod_{\alpha \in  Q'} \left( 1-x^{\alpha} \right)^{-\rho^{(m)}(\alpha)}.$$
    
Now we use the standard power series identity 
$(1-t)^{-N} = \sum_{n \geq 0} \binom{n+N-1}{N-1} t^n$ to see that
        \begin{align*}
        & \prod_{\alpha \in  Q'} \left( 1-x^{\alpha} \right)^{-\rho^{(m)}(\alpha)} \\
        &= \prod_{\alpha \in  Q'} \left( \sum_{n \geq 0} \binom{n+\rho^{(m)}(\alpha)-1}{\rho^{(m)}(\alpha)-1}  x^{n \alpha}\right) \\
        &= 1+ \sum_{\substack{\phi: Q' \rightarrow \mathbb{Z}_{\geq 0} \\ \phi(\alpha) = 0 ~\text{a.e.}}} x^{\sum_{\alpha \in Q' } \phi(
        \alpha)\alpha} \prod_{\alpha\in Q'} \binom{\phi(\alpha)+\rho^{(m)}(\alpha)-1}{\rho^{(m)}(\alpha)-1} \\
        &= 1 + \sum_{\beta \in Q'} \left( \sum_{\substack{\phi: Q' \rightarrow \mathbb{Z}_{\geq 0} \\ \sum_{\alpha \in Q' } \phi(\alpha) \alpha = \beta}} \prod_{\alpha\in Q'} \binom{\phi(\alpha)+\rho^{(m)}(\alpha)-1}{\rho^{(m)}(\alpha)-1} \right) x^{\beta} \\
        &= m_{\emptyset}(X) + \sum_{\lambda \in \Par \setminus \{\emptyset \}} \left( \sum_{\substack{\phi: Q' \rightarrow \mathbb{Z}_{\geq 0} \\ \sum_{\alpha \in Q' } \phi(\alpha) \alpha = \lambda}} \prod_{\alpha\in Q'} \binom{\phi(\alpha)+\rho^{(m)}(\alpha)-1}{\rho^{(m)}(\alpha)-1} \right) m_{\lambda}(X).\\
    \end{align*}
\end{proof}

We may interpret the values $\rho^{(m)}(\lambda)$ in terms of hyper-partitions using the following result. 

\begin{thm}
    Suppose $M$ is a finite set on which $\mathfrak{S}_n$ acts by permutations. Then the Frobenius characteristic $\mathrm{Frob}_{\mathfrak{S}_n}(\mathbb{C}[M])$ has the monomial expansion 
    $$\mathrm{Frob}_{\mathfrak{S}_n}(\mathbb{C}[M]) = \sum_{\lambda \vdash n} |M/\mathfrak{S}_{\lambda}| m_{\lambda}(X)$$ where $M/\mathfrak{S}_{\lambda}$ denotes the set of orbits in $M$ of the Young subgroup $\mathfrak{S}_{\lambda} = \mathfrak{S}_{\lambda_1}\times \ldots \times \mathfrak{S}_{\lambda_{\ell}}.$
\end{thm}
\begin{proof}
    This follows from \textbf{Frobenius reciprocity} between $\mathfrak{S}_n$ and the subgroups $\mathfrak{S}_{\lambda}.$ We refer the reader to Dotsenko's article \cite{Dotsenko}.
\end{proof}

As an immediate consequence, we find that the numbers $\rho^{(m)}(\lambda)$ have the combinatorial interpretation 
$$\rho^{(m)}(\lambda) = | \pi_{|\lambda|}^{(m)}/\mathfrak{S}_{\lambda}|.$$ This implies that the recursion Proposition \ref{monomial recursion prop} should be combinatorially meaningful. It would be interesting to give a fully combinatorial description of this formula.

\subsection{Power sum expansion}

Here we will find a recursive formula for the power sum expansion coefficients of the Bell functions $B_n^{(m)}.$ It will be convenient to introduce the following notation.

\begin{defn}
    For $\lambda, \mu^{(1)},\ldots,\mu^{(n)} \in \Par$ and $\lambda = \mu^{(1)}\cup\cdots\cup \mu^{(n)} $ define 
    $ \binom{\lambda}{\mu^{(1)},\ldots,\mu^{(n)}}:= \prod_{j \geq 1} \frac{m_j(\lambda)!}{m_j(\mu^{(1)})!\cdots m_j(\mu^{(n)})!}.$ Define $\beta^{(m)}(\lambda) \in \mathbb{Q}$ as 
    $B_n^{(m)} = \sum_{\lambda \vdash n} \frac{\beta^{(m)}(\lambda)}{z_{\lambda}} p_{\lambda}.$
\end{defn}

Note that by Theorem \ref{Frob thm} we know that for all $\lambda \neq \emptyset,$
$\beta^{(m)}(\lambda) = \mathrm{fix}(\pi_{|\lambda|}^{(m)},\lambda)$, i.e., the number of order $m$ hyper-partitions of $\{1,\ldots,|\lambda|\}$ fixed by any permutation with cycle type $\lambda.$ In particular, $\beta^{(m)}(\lambda) \in \mathbb{Z}_{\geq 0}.$ We will now prove the following recursive formula:

\begin{prop}\label{power sum recursion prop}
For all $m \geq 0$ and $\lambda \in \Par \setminus \{\emptyset \},$
    $$\beta^{(m+1)}(\lambda) = \sum_{n \geq 1} \frac{1}{n!} \sum_{\substack{\mu^{(1)}\cup\cdots \cup \mu^{(n)} = \lambda \\ \mu^{(i)} \neq \emptyset}} \binom{\lambda}{\mu^{(1)},\ldots,\mu^{(n)}}\prod_{1\leq i\leq n } \left( \sum_{d\mid \gcd(\mu^{(i)})} d^{\ell(\mu^{(i)})-1} \beta^{(m)}(\mu^{(i)}/d) \right).$$
\end{prop}
\begin{proof}
We proceed by induction:
    \begin{align*}
        & \Omega\left( \Omega_0^{(m+1)}(X)  \right) \\
        &=\Omega \left( \sum_{\mu \neq \emptyset} \frac{\beta^{(m)}(\mu)}{z_{\mu}} p_{\mu}(X) \right) \\
        &= \exp\left( \sum_{n \geq 1} \frac{1}{n} p_n \left( \sum_{\mu\neq \emptyset} \frac{\beta^{(m)}(\mu)}{z_{\mu}} p_{\mu} \right) \right)  \\
        &= \exp \left( \sum_{\mu \neq \emptyset} \left( \sum_{d \mid \gcd(\mu)} \frac{\beta^{(m)}(\mu/d)}{dz_{\mu/d}} \right) p_{\mu} \right)\\
    \end{align*}

    It may be checked directly that for all $\mu \in \Par$ with $d \mid \gcd(\mu),$
    $z_{\mu/d} = \frac{z_{\mu}}{d^{\ell(\mu)}}.$ Thus,
    \begin{align*}
    & \exp \left( \sum_{\mu \neq \emptyset} \left( \sum_{d \mid \gcd(\mu)} \frac{\beta^{(m)}(\mu/d)}{dz_{\mu/d}} \right) p_{\mu} \right) \\
        &= \exp \left( \sum_{\mu \neq \emptyset} \left( \sum_{d \mid \gcd(\mu)}  d^{\ell(\mu)-1}\beta^{(m)}(\mu/d)\right) \frac{p_{\mu}}{z_{\mu}} \right) \\
        &= \sum_{n \geq 0} \frac{1}{n!} \left( \sum_{\mu \neq \emptyset} \left( \sum_{d \mid \gcd(\mu)}  d^{\ell(\mu)-1}\beta^{(m)}(\mu/d)\right) \frac{p_{\mu}}{z_{\mu}} \right)^n\\
        &= 1 + \sum_{n \geq 1} \frac{1}{n!} \sum_{\mu^{(1)},\ldots,\mu^{(n)}\neq \emptyset } \frac{p_{\mu^{(1)}\cup\cdots \cup \mu^{(n)}}}{z_{\mu^{(1)}}\cdots z_{\mu^{(n)}}} \prod_{1\leq i\leq n } \left( \sum_{d\mid \gcd(\mu^{(i)})} d^{\ell(\mu^{(i)})-1} \beta^{(m)}(\mu^{(i)}/d) \right) \\
    \end{align*}

    Lastly, it may be checked that  For $\lambda, \mu^{(1)},\ldots,\mu^{(n)} \in \Par$ with $\lambda = \mu^{(1)}\cup\cdots\cup \mu^{(n)} $, 
    $\binom{\lambda}{\mu^{(1)},\ldots,\mu^{(n)}} = \frac{z_{\lambda}}{z_{\mu^{(1)}}\cdots z_{\mu^{(n)}}}.$
    Therefore, we arrive at
    $$ 1+ \sum_{\lambda \neq \emptyset } \left( \sum_{n \geq 1} \frac{1}{n!} \sum_{\substack{\mu^{(1)}\cup \cdots \cup\mu^{(n)} = \lambda \\ \mu^{(i)}\neq \emptyset}}\binom{\lambda}{\mu^{(1)},\ldots,\mu^{(n)}} \prod_{1\leq i\leq n } \left( \sum_{d\mid \gcd(\mu^{(i)})} d^{\ell(\mu^{(i)})-1} \beta^{(m)}(\mu^{(i)}/d) \right)  \right) \frac{p_{\lambda}}{z_{\lambda}}.$$
\end{proof}

\begin{remark}

Interestingly, since the numbers $\beta^{(m)}(\lambda)$ are integers, the denominators $\frac{1}{n!}$ in the recursion in Proposition \ref{power sum recursion prop} must cancel. Clearing the denominators gives 
$$\ell(\lambda)!\beta^{(m+1)}(\lambda) = \sum_{1\leq n \leq \ell(\lambda)} \frac{\ell(\lambda)!}{n!}\sum_{\substack{\mu^{(1)}\cup\cdots \cup \mu^{(n)} = \lambda \\ \mu^{(i)} \neq \emptyset}} \binom{\lambda}{\mu^{(1)},\ldots,\mu^{(n)}}\prod_{1\leq i\leq n } \left( \sum_{d\mid \gcd(\mu^{(i)})} d^{\ell(\mu^{(i)})-1} \beta^{(m)}(\mu^{(i)}/d) \right)$$ where both sides of the equation are non-negative integers. It would be interesting to give a combinatorial interpretation for this recursion. 

\end{remark}

\section{Bell symmetric functions}

In this section, we focus on the $m = 1$ case of the higher order Bell symmetric functions. We will simplify our notation in this case and write $B_n:= B^{(1)}_n$. 

\subsection{Some expansions}

\begin{example}
    Here are a few of the Bell symmetric functions:
    \begin{enumerate}
        \item $B_0  = s_{\emptyset}$
        \item $B_1  = s_{1}$
        \item $B_2  = 2s_2$
        \item $B_3 = 3s_3 + s_{2,1}$
        \item $B_4  = 5s_{4} + 2s_{2, 2} + 2s_{3, 1}$
        \item $B_5 = 7s_{5} + 5s_{4,1} + 4s_{3,2} + s_{2,2,1}$
    \end{enumerate}
    Note that $B_5 = 2h_5 + 2h_{4,1} + 3h_{3,2}-h_{3,1,1} + h_{2,2,1}$ is not $h$-positive.
\end{example}

Using Proposition \ref{plethystic recursion}, we see the following:

\begin{cor}
For all $n\geq 0,$
    $$B_n = \sum_{\lambda \vdash n} \prod_{j \geq 1} h_{m_j(\lambda)}(h_j)$$
\end{cor}

Next we will give the monomial and power sum expansions of the Bell functions $B_n.$

\begin{defn}
     Given $\lambda \in \mathbb{Y}$ define $\rho(\lambda)$ as the number of partitions of $\lambda$ in $Q$, that is, 
    the number of multisets $\{\alpha_1,\ldots , \alpha_m\}$ of $\alpha_j \in Q'$ such that $\lambda = \alpha_1+\ldots + \alpha_m.$
\end{defn}

Using Proposition \ref{monomial recursion prop}, we find that $\rho(\lambda) = \rho^{(1)}(\lambda).$ Thus as a result:

\begin{cor}
    $$B_n = \sum_{\lambda \vdash n} \rho(\lambda) m_{\lambda}$$
\end{cor}

Therefore, in particular, we find that for $\lambda \vdash n$ with $n \geq 1$, the number of $\mathfrak{S}_{\lambda}$ orbits in the set partitions of $\{1,\ldots, n\}$ is given by the number $\rho(\lambda)$ of vector partitions of $\lambda.$

\begin{example}
Below are the monomial expansions of the first few Bell symmetric functions:
    \begin{enumerate}
        \item $B_0 = m_{\emptyset} $
        \item $B_1 = m_{1}$
        \item $B_2 = 2m_{1,1} +2m_2$
        \item $B_3 = 5m_{1, 1, 1} + 4m_{2, 1} + 3m_{3}$
        \item $B_4 = 15m_{1, 1, 1, 1} + 11m_{2, 1, 1} + 9m_{2, 2} + 7m_{3, 1} + 5m_{4}$
        \item $B_5 = 52m_{1, 1, 1, 1, 1} + 36m_{2, 1, 1, 1} + 26m_{2, 2, 1} + 21m_{3, 1, 1} + 16m_{3, 2} + 12m_{4, 1} + 7m_{5}$
    \end{enumerate}
    Notice that the coefficient of $m_{1^n}$ in $B_n$ is exactly $b_n:=b_n^{(1)}$ whereas the coefficient of $m_n$ is the number of integer partitions of $n.$ This agrees with the formula $\rho(1^n) = b_n$ since the vector partitions of $(1,\ldots,1,0,\ldots)$ correspond exactly to set partitions of $\{1,\ldots,n\}$ and similarly $\rho( (n) ) = \#(\lambda \vdash n)$ since the vector partitions of $ (n,0,\ldots)$ correspond exactly to the integer partitions of $n.$
\end{example}

It will be useful to introduce the following notation.

\begin{defn}
$$\sigma(\mu^{(1)},\ldots , \mu^{(n)}):= \prod_{i =1}^{n} \sigma_{\ell(\mu^{(i)})-1}(\gcd(\mu^{(i)}))$$ where
    $\sigma_r(m):= \sum_{d\mid m} d^r.$
    For $\lambda \in \Par$ define $\beta(\emptyset):= 1$ and for $\lambda \neq \emptyset$
    $$\beta(\lambda):= \sum_{n \geq 1} \frac{1}{n!} \sum_{\substack{\mu^{(1)}\cup\cdots\cup \mu^{(n)} = \lambda  
 \\ \mu^{(i)} \neq \emptyset}} \sigma(\mu^{(1)},\ldots,\mu^{(n)})\binom{\lambda}{\mu^{(1)},\ldots,\mu^{(n)}}.$$
\end{defn}

From Proposition \ref{power sum recursion prop} we find:

\begin{cor}
For all $n \geq 1,$
    $$B_n = \sum_{\lambda \vdash n} \frac{\beta(\lambda)}{z_{\lambda}} p_{\lambda}.$$
\end{cor}

\begin{example}
    Below are the power sum expansions of the first few Bell symmetric functions:
    \begin{enumerate}
        \item $B_0 = p_{\emptyset}$
        \item $B_1 = p_{1}$
        \item $B_2 = p_2 + p_{1,1}$
        \item $B_3 = \frac{2}{3} p_3 + \frac{3}{2} p_{2,1} + \frac{5}{6} p_{1,1,1}$
        \item $B_4 = \frac{3}{4}p_4 + p_{3,1} + \frac{7}{8} p_{2,2} + \frac{7}{4} p_{2,1,1} + \frac{5}{8} p_{1,1,1,1} $
        \item $B_5 = \frac{2}{5} p_5 + p_{4,1} + \frac{5}{6} p_{3,2}+\frac{7}{6} p_{3,1,1} + \frac{3}{2}p_{2,2,1} + \frac{5}{3}p_{2,1,1,1}+\frac{13}{30} p_{1,1,1,1,1}$
    \end{enumerate}
\end{example}

\subsection{Recurrence relation}

The Bell numbers $b_n$ are well-known to satisfy the recursion
$b_{n+1} = \sum_{k=0}^{n} \binom{n}{k} b_k$ for all $n \geq 0.$ This may be seen by applying the chain rule to the power series composition $\exp( \exp(t)-1).$ In this section, we will prove an analogous statement for the Bell functions $B_n.$ We will need to introduce the following auxiliary symmetric functions.

\begin{defn}
For $n \geq 1$ define
    $$H_n:= \sum_{\lambda \vdash n} \frac{\sigma_{\ell(\lambda)-1}(\gcd(\lambda))}{z_{\lambda}} p_{\lambda}.$$
\end{defn}

The $H$ functions arise in the following expansion:

\begin{lem}\label{H function lem}
    $$\sum_{n \geq 1} \frac{1}{n} p_n\left( \Omega_0(tX) \right) = \sum_{n \geq 1} H_n t^n$$
\end{lem}
\begin{proof}
    First, expand the left hand side as 
    $$\sum_{n\geq 1} \frac{1}{n} p_n \left( \Omega_0(tX) \right) = \sum_{n \geq 1} \frac{1}{n} p_n \left( \sum_{\lambda \neq \emptyset} t^{|\lambda|} \frac{p_{\lambda}(X)}{z_{\lambda}} \right) = \sum_{\substack{n\geq 1 \\ \lambda \neq \emptyset}} t^{n|\lambda|}\frac{p_{n\lambda}(X)}{nz_{\lambda}} = \sum_{\lambda \neq \emptyset} \left( \sum_{d \mid \gcd(\lambda)}\frac{1}{dz_{\lambda/d}}\right) t^{|\lambda|} p_{\lambda}(X).$$ Lastly, using $z_{\lambda/d} = \frac{z_{\lambda}}{d^{\ell(\lambda)}}$ we see that $\sum_{d \mid \gcd(\lambda)}\frac{1}{dz_{\lambda/d}} = \frac{\sigma_{\ell(\lambda)-1}(\gcd(\lambda))}{z_{\lambda}}$ as required.
\end{proof}

The $H$ functions have a simple monomial expansion. Recall that $\sigma_{-1}(n) := \sum_{d \mid n} \frac{1}{d}.$

\begin{prop}
For all $n \geq 1,$
    $$H_n = \sum_{\lambda \vdash n} \sigma_{-1}(\gcd(\lambda)) m_{\lambda}.$$
\end{prop}
\begin{proof}
    Using Lemma \ref{H function lem} we may expand $\sum_{n\geq 1} \frac{1}{n} p_n \left( \Omega_0(tX) \right)$ as 
    $$\sum_{n\geq 1} \frac{1}{n} p_n \left( \Omega_0(tX) \right) = \sum_{n\geq 1} \frac{1}{n} p_n \left( \sum_{\lambda \neq \emptyset} t^{|\lambda|} m_{\lambda}(X) \right) = \sum_{\substack{n \geq 1 \\ \lambda \neq \emptyset}} \frac{t^{n|\lambda|}m_{n\lambda}(X)}{n}.$$ Rearranging this series gives 
    $$\sum_{\substack{n \geq 1 \\ \lambda \neq \emptyset}} \frac{t^{n|\lambda|}m_{n\lambda}(X)}{n} = \sum_{\lambda \neq \emptyset} \left( \sum_{d\mid \gcd(\lambda)} \frac{1}{d} \right) t^{|\lambda|} m_{\lambda}(X) = \sum_{\lambda \neq \emptyset} \sigma_{-1}(\gcd(\lambda)) t^{|\lambda|} m_{\lambda}(X).$$
\end{proof}

\begin{example}
    Here are a few of the $H$-functions:
    \begin{enumerate}
        \item $H_0 = m_{\emptyset}$
        \item $H_1 = m_{1}$
        \item $H_{2} = \frac{3}{2}m_{2} + m_{1,1}$
        \item $H_3 = \frac{4}{3} m_{3} + m_{2,1} + m_{1,1,1}$
        \item $H_4 = \frac{7}{4} m_{4} + m_{3,1} + \frac{3}{2} m_{2,2} + m_{2,1,1} + m_{1,1,1,1}$
        \item $H_5 = \frac{6}{5} m_{5} + m_{4,1} + m_{3,2} + m_{3,1,1} + m_{2,2,1} + m_{2,1,1,1}+m_{1,1,1,1,1}$
        \item $H_6 = 2m_6 + m_{5,1} + \frac{3}{2}m_{4,2}+ m_{4,1,1} + \frac{4}{3}m_{3,3} + m_{3,2,1} + \frac{3}{2}m_{2,2,2}+m_{2,2,1,1}+m_{2,1,1,1,1}+m_{1,1,1,1,1,1}$
    \end{enumerate}
Note that $H_n$ depends heavily on the divisors of $n.$ In particular, for a prime $p$, 
$H_p = \frac{1}{p} m_{p} + h_p.$

\end{example}

\begin{prop}
For all $n \geq 0,$
    $$B_{n+1} = \sum_{k=0}^{n} \frac{n-k+1}{n+1} H_{n-k+1} B_k.$$
\end{prop}
\begin{proof}
We begin by applying the chain rule:
    \begin{align*}
        &\frac{d}{dt} \Omega\left( \Omega_0(tX) \right) \\
        &= \frac{d}{dt} \exp\left( \sum_{n \geq 1} \frac{1}{n} p_n\left( \Omega_0(tX) \right) \right) \\
        &= \exp\left( \sum_{n \geq 1} \frac{1}{n} p_n\left( \Omega_0(tX) \right) \right)  \frac{d}{dt} \left( \sum_{n \geq 1} \frac{1}{n} p_n\left( \Omega_0(tX) \right) \right) \\
        &= \Omega\left( \Omega_0(tX) \right) \left(\frac{d}{dt} \sum_{n \geq 1} H_n(X) t^n \right) \\
        &= \Omega\left( \Omega_0(tX) \right) \left(\sum_{n \geq 1} nH_n(X) t^{n-1} \right) \\
        &= \left(\sum_{n \geq 0}B_n t^n \right) \left( \sum_{n \geq 0} (n+1)H_{n+1}(X)t^{n}  \right)  \\ 
        &=  \sum_{n \geq 0} \left( \sum_{k=0}^{n} (n-k+1)H_{n-k+1}(X)B_k(X) \right) t^n  \\
    \end{align*}

    On the other hand, 
    $$\frac{d}{dt} \Omega\left( \Omega_0(tX) \right) = \frac{d}{dt} \sum_{n \geq 0} B_n(X) t^n = \sum_{n \geq 0} (n+1)B_{n+1}(X) t^{n}$$ and therefore,
    $$(n+1)B_{n+1} = \sum_{k=0}^{n} (n-k+1)H_{n-k+1}B_k.$$ By dividing by $n+1$ we get the result.
\end{proof}

\subsection{Relation to restriction coefficients}

In this last section, we show that the Schur expansion coefficients of the Bell symmetric functions $B_n$ are directly related to the \textbf{restriction coefficients}. 

\begin{defn}
For a partition $\lambda$ and $n \geq 0$ let $\mathbb{S}^{\lambda}\mathbb{C}^n$ denote the $\mathrm{GL}_{n}(\mathbb{C})$ module obtained by applying the Schur functor $\mathbb{S}^{\lambda}$ to the standard representation $\mathbb{C}^n$ of $\mathrm{GL}_{n}(\mathbb{C})$. We embed $\mathfrak{S}_n \subset \GL_n(\mathbb{C})$ as the set of \textit{\textbf{permutation matrices}}. For $|\mu| = n$ let $V_{\mu}$ denote the irreducible Specht module of $\mathfrak{S}_n$ over $\mathbb{C}$. For $\lambda,\mu \in \Par$ with $n = |\mu|$ define the restriction coefficient $r_{\lambda}^{\mu}$ as 
    $$r_{\lambda}^{\mu}:= \dim_{\mathbb{C}} \Hom_{\mathfrak{S}_n}\left( V_{\mu} , \mathrm{Res}^{\GL_n(\mathbb{C})}_{\mathfrak{S}_n} \mathbb{S}^{\lambda} \mathbb{C}^n \right).$$ Define the (single-variable) \textit{\textbf{restriction series}} $R_{\lambda}(z) \in \mathbb{Z}_{\geq 0}[[z]]$ by 
    $$R_{\lambda}(z):= \sum_{n \geq 0} r_{\lambda}^{(n)}z^{n}.$$
\end{defn}

Note that $r_{\lambda}^{(n)} = \dim_{\mathbb{C}} \left( \mathbb{S}^{\lambda} \mathbb{C}^n \right)^{\mathfrak{S}_n}$ is the dimension of the $\mathfrak{S}_n$-invariants of $\mathbb{S}^{\lambda} \mathbb{C}^n.$ For $|\mu| = n$ let $c_{\mu} \in \mathbb{C}[\mathfrak{S}_{n}]$ denote Young symmetrizer corresponding to $\mu$. Then for $|\lambda| = k$, $r_{\lambda}^{(n)} = \dim_{\mathbb{C}} c_{1^n} (\mathbb{C}^n)^{\otimes k} c_{\lambda}$ where $\mathfrak{S}_{n}$ acts on the left and $\mathfrak{S}_{k}$ on the right of $(\mathbb{C}^n)^{\otimes k}.$ For details regarding the special case of the restriction coefficients of the form $r_{\lambda}^{(n)}$ we refer the reader to Narayanan--Paul--Prasad--Srivastava \cite{NSPP21}. The restriction coefficients $r_{\lambda}^{\mu}$ may be understood algebraically using Littlewood's reciprocity theorem:

\begin{thm}[Littlewood \cite{Littlewood58}]\label{tauberian thm}
For $\lambda,\mu \in \Par,$
    $$r_{ \lambda}^{\mu} = \langle s_{\lambda},s_{\mu}(\Omega(X)) \rangle.$$
\end{thm}

Using Littlewood's reciprocity result, we find a convenient algebraic formula for the restriction series $R_{\lambda}(z):$

\begin{lem}\label{restriction lemma}
    $$\Omega\left(z \Omega(X) \right) = \sum_{\lambda} R_{\lambda}(z) s_{\lambda}(X)$$
\end{lem}
\begin{proof}
Using Theorem \ref{tauberian thm} we find
    \begin{align*}
        \Omega\left(z \Omega(X) \right) &= \sum_{n \geq 0} h_n(z\Omega(X)) \\
        &= \sum_{n \geq 0} z^{n} h_{n}(\Omega(X))\\ 
        &= \sum_{n \geq 0} z^{n} \sum_{\lambda} r_{\lambda}^{(n)} s_{\lambda}(X)\\
        &= \sum_{\lambda} \left( \sum_{n \geq 0} z^{n} r_{\lambda}^{(n)} \right) s_{\lambda}(X).\\
    \end{align*}

\end{proof}

We will use the following notation for the Schur expansion coefficients of the Bell functions $B_n.$

\begin{defn}
    Define $A_{\lambda} \in \mathbb{Z}_{\geq 0}$ via 
    $B_n = \sum_{\lambda \vdash n} A_{\lambda} s_{\lambda}.$
\end{defn}

In order to relate the Schur coefficients $A_{\lambda}$ to the restriction coefficients $r_{\lambda}^{\mu}$, we will need to prove a simple result regarding the restriction series $R_{\lambda}(z).$

\begin{prop}
    For all $\lambda \in \Par,$
     $(1-z)R_{\lambda}(z)$ is a polynomial with degree at most $|\lambda|$. 
\end{prop}
\begin{proof}
Using Lemma \ref{restriction lemma} we see that 
$$R_{\lambda}(z) = \langle s_{\lambda}(X), \Omega\left(z \Omega(X) \right) \rangle.$$ We will use the expansion $s_{\lambda} = \sum_{\mu \vdash n} K^{(-1)}_{\lambda,\mu} h_{\mu}$ where $K^{(-1)}_{\lambda,\mu}$ are the \textbf{inverse Kostka numbers} and instead compute $\langle h_{\mu}(X), \Omega\left(z \Omega(X) \right) \rangle$. Following similar to the proof of Proposition \ref{monomial recursion prop}, 

\begin{align*}
    &\Omega\left(z\Omega(X) \right) \\
    &= \Omega\left( \sum_{\alpha \in Q} zx^{\alpha} \right)\\
    &= \prod_{\alpha \in Q} \frac{1}{1-zx^{\alpha}}\\
    &= \frac{1}{1-z}\prod_{\alpha \in Q\setminus \{0\}} \frac{1}{1-zx^{\alpha}}\\
    &= \frac{1}{1-z} \sum_{\alpha \in Q} \left( \sum_{\substack{\phi:Q'\rightarrow \mathbb{Z}_{\geq 0} \\ \sum_{\beta \in Q\setminus \{0\}} \phi(\beta)\beta = \alpha } } z^{\sum_{\beta \in Q\setminus \{0\}}\phi(\beta)}  \right) x^{\alpha }\\
    &= \frac{1}{1-z} \sum_{\mu} \left( \sum_{\substack{\phi:Q'\rightarrow \mathbb{Z}_{\geq 0} \\ \sum_{\beta \in Q\setminus \{0\}} \phi(\beta)\beta = \mu } } z^{\sum_{\beta \in Q\setminus \{0\}}\phi(\beta)}  \right) m_{\mu}(X).\\
\end{align*}

Define the series
$$f_{\mu}(z):= \sum_{\substack{\phi:Q'\rightarrow \mathbb{Z}_{\geq 0} \\ \sum_{\beta \in Q\setminus \{0\}} \phi(\beta)\beta = \mu } } z^{\sum_{\beta \in Q\setminus \{0\}}\phi(\beta)}.$$ Note that $f_{\mu}(z)$ is a polynomial with degree $|\mu|.$ The leading coefficient of $z^{|\mu|}$ is always $1$ corresponding to the vector partition $\mu= \mu_1(1,0\ldots) + \mu_2 (0,1,0,\ldots) + \ldots + \mu_{\ell(\mu)}(0,\ldots,0,1,0,\ldots)$. Now we see 
$$\Omega\left(z\Omega(X) \right) = \sum_{\mu} \frac{f_{\mu}(z)}{1-z} m_{\mu}(X).$$ This means that 
$$\langle h_{\mu}(X), \Omega\left(z\Omega(X) \right) \rangle = \frac{f_{\mu}(z)}{1-z}.$$ Putting this all together, 
\begin{align*}
    &R_{\lambda}(z) = \langle s_{\lambda}(X), \Omega\left(z \Omega(X) \right) \rangle\\
    &= \sum_{\mu\vdash |\lambda|} K^{(-1)}_{\lambda,\mu} \langle h_{\mu}(X), \Omega\left(z \Omega(X) \right) \rangle\\
    &= \sum_{\mu\vdash |\lambda|} K^{(-1)}_{\lambda,\mu} \frac{f_{\mu}(z)}{1-z}\\
    &= \frac{\sum_{\mu\vdash |\lambda|} K^{(-1)}_{\lambda,\mu} f_{\mu}(z)}{1-z}.\\
\end{align*}

It is clear that $\sum_{\mu\vdash n} K^{(-1)}_{\lambda,\mu} f_{\mu}(z)$ is a polynomial with degree at most $|\lambda|.$

\end{proof}

The Schur coefficients $A_{\lambda}$ are the $z = 1$ residues of the restriction series $R_{\lambda}(z).$

\begin{prop}\label{Bell coefficients from restriction}
    $$A_{\lambda} = \lim_{z \rightarrow 1} (1-z)R_{\lambda}(z)$$
\end{prop}
\begin{proof}
Notice that
    $$\Omega(z\Omega_0(tX)) = \Omega(z\Omega(tX) -z) = (1-z)\Omega(z\Omega(tX)).$$ The left hand side limits to 
    $$\Omega(\Omega_0(tX)) = \sum_{n \geq 0} B_n(X)t^n$$ whereas the right hand side is given by 
    $$\sum_{\lambda} \left(\lim_{z \rightarrow 1}(1-z)R_{\lambda}(z) \right) s_{\lambda}(X)t^{|\lambda|}.$$
\end{proof}

The main result of this section is an application of a weak version of the Hardy--Littlewood \textbf{Tauberian theorem} which we now review.

\begin{thm}[Hardy--Littlewood]\label{tauberian}
Suppose $a_n \geq 0$ for all $n \geq 0$ and $\sum_{n \geq 0} a_n z^n \sim \frac{1}{1-z}$ as $z \uparrow 1.$ Then as $n \rightarrow \infty$ we have 
$ \sum_{k\leq n} a_k \sim n.$
\end{thm}

We may apply the Tauberian theorem directly. 

\begin{thm}\label{Schur asymptotic thm}
    $$A_{\lambda} = \lim_{n \rightarrow \infty} \frac{1}{n} \sum_{k\leq n} r_{\lambda}^{(k)}$$
\end{thm}
\begin{proof}
Suppose first that $A_{\lambda} \neq 0.$ We have that 
    $R_{\lambda}(z) = \sum_{n\geq 0} r_{\lambda}^{(n)} z^n$ so by Proposition \ref{Bell coefficients from restriction} we know $R_{\lambda}(z) \sim \frac{A_{\lambda}}{1-z}$ as $z \uparrow 1.$ Therefore, by Theorem \ref{tauberian} we find that $\frac{1}{n}\sum_{k \leq n} r_{\lambda}^{(k)} \sim A_{\lambda} $ as $n \rightarrow \infty.$ In the case that $A_{\lambda} = 0,$ the polynomial $(1-z)R_{\lambda}(z)$ has $z=1$ as a root and hence $R_{\lambda}(z)$ must itself be a polynomial. But then, $r_{\lambda}^{(k)}$ is only nonzero for finitely many $k$ and thus $\lim_{n\rightarrow \infty} \frac{1}{n}\sum_{k \leq n} r_{\lambda}^{(k)} = 0 = A_{\lambda}$ as required.
\end{proof}

Notice that $\rho(\mu) = f_{\mu}(1)$ so $A_{\lambda} = \sum_{\mu \vdash |\lambda|} K_{\lambda,\mu}^{(-1)} \rho(\mu).$ This sum has a (signed) combinatorial expansion using the results of Eğecioğlu–-Remmel on the inverse Kostka numbers \cite{ER90}. This leads to the pleasant formula
    $$\lim_{n \rightarrow \infty} \frac{1}{n} \sum_{k \leq n} r_{\lambda}^{(k)} = \sum_{\mu \vdash |\lambda|} K_{\lambda,\mu}^{(-1)} \rho(\mu).$$

\begin{remark}
    We may generalize the restriction series $R_{\lambda}(z)$ to a symmetric function $R_{\lambda}(Z) = R_{\lambda}(z_1,z_2,\ldots)$ by defining 
    $$R_{\lambda}(Z):= \langle s_{\lambda}(X), \Omega(Z\Omega(X)) \rangle.$$ The symmetric series $R_{\lambda}$ have previously been considered by Lee \cite{Lee24}. Using Lee's conventions, $R_{\lambda}(Z) = \mathcal{F}\{s_{\lambda} \}(Z)$ where $\mathcal{F}\{ - \}$ is the \textbf{Frobenius transform}. It is straightforward to check that $R_{\lambda}(Z) = \sum_{\mu} r_{\lambda}^{\mu} s_{\mu}(Z)$ and furthermore, $\Omega(-Z)R_{\lambda}(Z) = \langle s_{\lambda}(X), \Omega(Z\Omega_0(X)) \rangle$ is an inhomogeneous symmetric function with maximal degree at most $|\lambda|.$ In particular, 
    $\Omega(-Z)R_{\lambda}(Z) = \sum_{\mu \vdash |\lambda|} K_{\lambda,\mu}^{(-1)} F_{\mu}(Z)$ where 
    $F_{\mu}(Z):= \langle h_{\mu}(X), \Omega(Z \Omega_0(X)) \rangle .$ For all $\mu$, $F_{\mu}(Z)$ is explicitly given by 
    $$F_{\mu}(Z) = \sum_{\Phi \in Y(\mu)} z^{|\Phi|}$$ where 
    $Y(\mu)$ is the set of all $\Phi: \mathbb{Z}_{\geq 1} \times Q\setminus \{0\} \rightarrow \mathbb{Z}_{\geq 0}$ such that $\Phi_i(\alpha) \neq 0$ for only finitely many pairs $(i,\alpha) \in \mathbb{Z}_{\geq 1} \times Q\setminus \{0\} $ and $\sum_{(i,\alpha) \in \mathbb{Z}_{\geq 1} \times Q\setminus \{0\}} \Phi_i(\alpha) \alpha = \mu$ and where $|\Phi| \in Q$ is the vector $|\Phi|:= \left( \sum_{\alpha \in Q\setminus \{0\} }\Phi_i(\alpha) \right)_{i \geq 1}.$ In finitely many variables, $z_1,\ldots, z_k,$ this means that 
    $R_{\lambda}(z_1,\ldots,z_k) = \sum_{\ell(\mu) \leq k} r_{\lambda}^{\mu} s_{\mu}(z_1,\ldots,z_k)$ has the form $$R_{\lambda}(z_1,\ldots,z_k)= \frac{\sum_{\mu \vdash |\lambda|} K_{\lambda,\mu}^{(-1)} F_{\mu}(Z)}{(1-z_1)\cdots (1-z_k)}$$ where the numerator is a symmetric polynomial in the variables $z_1,\ldots,z_k$ with total degree at most $|\lambda|.$ It would be interesting to see if any \textbf{multi-variable Tauberian theorems} could be applied here to obtain other types of asymptotics for the restriction coefficients. 
\end{remark}

\printbibliography

\end{document}